\documentclass[11pt, a4paper]{article}

\usepackage{amsmath}
\usepackage{amssymb}
\usepackage{amsthm}
\usepackage[linktoc = all, hidelinks]{hyperref} 
\usepackage[nameinlink, noabbrev]{cleveref} 
\usepackage{enumitem}
\usepackage{float}
\usepackage{subcaption}
\usepackage[margin = 2cm]{geometry}

\usepackage{pgf,tikz}
\usepackage{mathrsfs}
\usetikzlibrary{decorations.pathreplacing}
\usetikzlibrary[patterns]
\usetikzlibrary{arrows.meta}
\tikzset{font={\fontsize{10pt}{12}\selectfont}}
\usepackage{tkz-euclide}
\tkzSetUpPoint[fill=black, size = 3pt]

\usepackage{thmtools}
\usepackage{thm-restate}

\declaretheoremstyle[
spaceabove = .7\baselineskip\@plus.2\baselineskip\@minus.2\baselineskip, 
spacebelow = .7\baselineskip\@plus.2\baselineskip\@minus.2\baselineskip,
headfont = \normalfont\bfseries,
notefont = \mdseries, 
notebraces = {(}{)},
bodyfont = \normalfont\itshape,
postheadspace = .5em,
headpunct = .
]{bolditalic}

\declaretheoremstyle[
spaceabove = .5\baselineskip\@plus.2\baselineskip\@minus.2\baselineskip, 
spacebelow = .5\baselineskip\@plus.2\baselineskip\@minus.2\baselineskip,
headfont = \normalfont\bfseries,
notefont = \mdseries, 
notebraces = {(}{)},
bodyfont = \normalfont,
postheadspace = .5em,
headpunct = .
]{boldnormal}

\declaretheoremstyle[
spaceabove = .5\baselineskip\@plus.2\baselineskip\@minus.2\baselineskip, 
spacebelow = .5\baselineskip\@plus.2\baselineskip\@minus.2\baselineskip,
headfont = \normalfont\itshape,
notefont = \mdseries, 
notebraces = {}{},
bodyfont = \normalfont,
postheadspace = .5em,
headpunct = .,
qed = \qedsymbol
]{proofstyle}

\declaretheorem[name = Claim, numbered = yes, style = bolditalic, refname = {claim,claims}, Refname = {Claim,Claims}]{claim}
\declaretheorem[name = Corollary, numberlike = claim, style = bolditalic, refname = {corollary,corollaries}, Refname = {Corollary,Corollaries}]{corollary}
\declaretheorem[name = Definition, numberlike = corollary, style = bolditalic, refname = {definition,definitions}, Refname = {Definition,Definitions}]{definition}
\declaretheorem[name = Lemma, numberlike = corollary, style = bolditalic, refname = {lemma,lemmata}, Refname = {Lemma,Lemmata}]{lemma}

\declaretheorem[name = Proposition, numberlike = corollary, style = bolditalic, refname = {proposition,propositions}, Refname = {Proposition,Propositions}]{proposition}
\declaretheorem[name = Theorem, numberlike = corollary, style = bolditalic, refname = {theorem,theorems}, Refname = {Theorem,Theorems}]{theorem}

\declaretheorem[name = Remark, numberlike = corollary, style = boldnormal, refname = {remark,remarks}, Refname = {Remark,Remarks}]{remark}

\title{Graphs with no induced $K_{2, t}$}
\author{Freddie Illingworth\thanks{DPMMS, University of Cambridge, UK. \emph{E-mail:} \textsf{fci21@cam.ac.uk}. Research supported by an EPSRC grant. \hfill \null \linebreak \emph{2020 MSC:} 05C35.}}
\date{}

\begin{document}

\maketitle

\begin{abstract}
	Consider a graph $G$ on $n$ vertices with $\alpha \binom{n}{2}$ edges which does not contain an induced $K_{2, t}$ ($t \geqslant 2$). How large must $\alpha$ be to ensure that $G$ contains, say, a large clique or some fixed subgraph $H$? We give results for two regimes: for $\alpha$ bounded away from zero and for $\alpha = o(1)$.
	
	Our results for $\alpha = o(1)$ are strongly related to the Induced Tur\'{a}n numbers which were recently introduced by Loh, Tait, Timmons and Zhou. For $\alpha$ bounded away from zero, our results can be seen as a generalisation of a result of Gy\'{a}rf\'{a}s, Hubenko and Solymosi and more recently Holmsen (whose argument inspired ours).
\end{abstract}

\section{Introduction}

Fix an integer $t \geqslant 2$ and consider a graph $G$ on $n$ vertices with $\alpha \binom{n}{2}$ edges which does not contain an induced $K_{2, t}$. How large does $\alpha$ have to be to ensure that $G$ contains some substructure (like a large clique or a fixed subgraph $H$)? We consider two regimes: $\alpha$ is bounded away from zero and $\alpha$ goes to zero as $n$ goes to infinity.

In the regime where $\alpha$ is bounded away from zero, $G$ will contain substructures that grow with $n$ (so for example the clique number of $G$, $\omega(G)$, will go to infinity). Gy\'{a}rf\'{a}s, Hubenko and Solymosi~\cite{GHS2002} dealt with the clique number in the case when $t = 2$ (that is, $G$ contains no induced $C_{4}$), confirming a conjecture of Erd\H{o}s.

\begin{proposition}[Gy\'{a}rf\'{a}s-Hubenko-Solymosi,~\cite{GHS2002}]\label{GHS}
	Let $G$ be a graph on $n$ vertices with $\alpha \binom{n}{2}$ edges. If $G$ does not contain an induced $K_{2, 2}$, then $\omega(G) \geqslant \alpha^{2} n/10$.
\end{proposition}

\noindent This was recently improved by Holmsen~\cite{Holmsen2019} (note that $1 - \sqrt{1 - \alpha} \geqslant \alpha/2$ for $\alpha \in [0, 1]$).

\begin{proposition}[Holmsen,~\cite{Holmsen2019}]\label{H}
	Let $G$ be a graph on $n$ vertices with $\alpha \binom{n}{2}$ edges. If $G$ does not contain an induced $K_{2, 2}$, then $\omega(G) \geqslant (1 - \sqrt{1 - \alpha})^{2}n$.
\end{proposition}

\noindent This result has the added advantage that $(1 - \sqrt{1 - \alpha})^{2} \rightarrow 1$ as $\alpha \rightarrow 1$, so it is approximately tight as $\alpha \rightarrow 1$. The arguments in this paper are motivated by Holmsen's.

Our main result is Theorem~\ref{main}, which is an extension to the situation where $G$ does not contain an induced $K_{2, t}$ and also considers whether $G$ contains some general subgraph (in place of a clique). For comparison with Proposition~\ref{H}, we state the special case of the clique (we believe this result is also in a sense tight as $\alpha \rightarrow 1$ -- see Remark~\ref{remark}). First, it will be convenient to define a constant $\beta$ depending on $\alpha$ and $t$.

\begin{definition}
	Given $\alpha \in [0, 1]$ and an integer $t \geqslant 2$, define
	\begin{equation*}
		\beta_{t}(\alpha) = \frac{t}{2 \sqrt{t - 1}} \Bigl[\sqrt{1 - \bigl(1 - \tfrac{2}{t}\bigr)^{2} \alpha} - \sqrt{1 - \alpha}\Bigr].
	\end{equation*}
\end{definition}

\noindent Note that $\beta_{2}(\alpha) = 1 - \sqrt{1 - \alpha}$ so Proposition~\ref{H} can be stated as: if $G$ is a graph on $n$ vertices with $\alpha \binom{n}{2}$ edges containing no induced $K_{2, 2}$, then $\omega(G) \geqslant \beta_{2}(\alpha)^{2} n$.

\begin{restatable}{theorem}{cliquethm}\label{cliquethm}
	Let $G$ be a graph on $n$ vertices with $\alpha \binom{n}{2}$ edges containing no induced $K_{2, t}$ and let $\beta = \beta_{t}(\alpha)$. For any positive integer $r$ with $R(t, r) \leqslant \beta^{2} n$, we have $\omega(G) \geqslant r + 1$.
\end{restatable}

\noindent Here $R(t, r)$ denotes the usual Ramsey number. It is natural for Ramsey numbers to appear in the statement. The class of graphs with ``no induced $K_{2, t}$'' includes those with ``no independent $t$-set'' and if $\omega(G) \geqslant r + 1$ for all such graphs, then $R(t, r + 1) \leqslant n$.

Since $R(2, r) = r$, Theorem~\ref{cliquethm} is exactly Holmsen's result when $t = 2$. In Section~\ref{cliques}, using known Ramsey number bounds we prove explicit lower bounds for the clique number for all $t$. As an illustration, we state the case $t = 3$, which is particularly clean.

\begin{theorem}\label{K23thm}
	Let $G$ be a graph on $n$ vertices with $\alpha \binom{n}{2}$ edges. If $G$ does not contain an induced $K_{2, 3}$, then
	\begin{align*}
		\omega(G) & \geqslant \bigl\lfloor \tfrac{2}{3} \alpha \sqrt{n} \bigr\rfloor \text{ for all } n, \text{ and} 
		\\
		\omega(G) & \geqslant \tfrac{1}{3} \alpha \sqrt{n \log n} + 2 \text{ for large enough }n \text{ in terms of } \alpha.
	\end{align*}
\end{theorem}

\noindent The regime where $\alpha$ goes to zero is closely related to the following natural question first proposed by Loh, Tait, Timmons and Zhou~\cite{LTTZ2017}. Consider a graph $G$ on $n$ vertices with $\alpha \binom{n}{2}$ edges containing no induced $K_{2, t}$ -- how large must $\alpha$ be to ensure that some fixed graph $H$ is a subgraph of $G$? If we do not ban $G$ from containing an induced $K_{2, t}$ then the answer follows from the theorem of Erd\H{o}s and Stone~\cite{ErdosStone1946} (see Erd\H{o}s and Simonovits~\cite{ErdosSimonovtis1966}): $\alpha = 1 - \frac{1}{\chi(H) - 1} + o(1)$ where $\chi(H)$ is the chromatic number of $H$. However forbidding $G$ from containing an induced $K_{2, t}$ (ruling out Tur\'{a}n-style graphs) changes the answer drastically. In particular we will see that the required $\alpha$ grows like $n^{-1/2}$, that is, the required number of edges grows like $n^{3/2}$.

Loh, Tait, Timmons and Zhou introduced the notion of an \emph{induced Tur\'{a}n number}: define
\begin{equation*}
	\operatorname{ex}(n, \{H, F\textnormal{-ind}\})
\end{equation*}
to be the maximum number of edges in a graph on $n$ vertices which does not contain $H$ as a subgraph and does not contain $F$ as an induced subgraph. In this paper we focus on $F = K_{2, t}$, which was also considered by Loh, Tait, Timmons and Zhou. We will give some improvements to their results. The important case where $H$ is an odd cycle has been resolved by Ergemlidze, Gy\H{o}ri and Methuku~\cite{EGM2019}.

\begin{proposition}[Loh-Tait-Timmons-Zhou,~\cite{LTTZ2017}]\label{LTTZclique}
	Let $t \geqslant 3$ be an integer and $G$ be a graph on $n$ vertices within minimum degree $d$. If $G$ does not contain an induced $K_{2, t}$, then
	\begin{equation*}
		\omega(G) \geqslant \biggl(\frac{d^{2}}{2n(t - 1)} (1 - o(1))\biggr)^{\frac{1}{t - 1}} - t + 1.
	\end{equation*}
\end{proposition}

\noindent A graph with $\alpha \binom{n}{2}$ edges has average degree $\alpha(n - 1)$ and has a subgraph of minimum degree at least $\alpha(n - 1)/2$. Thus one should view $d$ as being between $\alpha(n - 1)/2$ and $\alpha(n - 1)$. We improve the dependence upon $t$ for all $\alpha$ as well as adding a $(\log n)^{1 - \frac{1}{t - 1}}$ factor for constant $\alpha > 0$.

\begin{theorem}\label{K2tthm}
	Let $G$ be a graph on $n$ vertices with $\alpha \binom{n}{2}$ edges. If $G$ does not contain an induced $K_{2, t}$, then
	\begin{align*}
		\omega(G) & \geqslant \bigl\lfloor \tfrac{t - 1}{4} (\alpha^{2} n)^{\frac{1}{t - 1}}\bigr\rfloor - t + 3 \text{ for all } n, \text{ and} \\
		\omega(G) & \geqslant \tfrac{1}{20 t} \bigl(\alpha^{2} n (\log n)^{t - 2}\bigr)^{\frac{1}{t - 1}} \text{ for large enough } n \text{ in terms of } \alpha.
	\end{align*}
\end{theorem}

\noindent Finally, Loh, Tait, Timmons and Zhou gave a general upper bound for $\operatorname{ex}(n, \{H, F\textnormal{-ind}\})$ when $F = K_{2, t + 1}$.

\begin{proposition}[Loh-Tait-Timmons-Zhou,~\cite{LTTZ2017}]\label{LTTZindtur}
	Fix a graph $H$ with $v_{H}$ vertices. For any integer $t \geqslant 2$,
	\begin{equation*}
		\operatorname{ex}(n, \{H, K_{2, t + 1}\textnormal{-ind}\}) < (\sqrt{2} + o(1)) t^{\frac{1}{2}} (v_{H} + t)^{\frac{t}{2}} n^{\frac{3}{2}}.
	\end{equation*}
\end{proposition}

\noindent They also noted that a corollary of F\"{u}redi~\cite{Furedi1996} is that, for $H$ not bipartite,
\begin{equation*}
	\tfrac{1}{4} t^{\frac{1}{2}} n^{\frac{3}{2}} - \mathcal{O}_{t}\bigl(n^{\frac{4}{3}}\bigr) \leqslant \operatorname{ex}(n, \{H, K_{2, t + 1}\textnormal{-ind}\}).
\end{equation*}
In particular, for non-bipartite $H$, $\operatorname{ex}(n, \{H, K_{2, t + 1}\textnormal{-ind}\}) = \Theta_{t}(n^{3/2})$ but the correct growth rate in $t$ lies between $\frac{1}{4} t^{1/2} n^{3/2}$ and $C_{H} t^{(t + 1)/2} n^{3/2}$. We give a slightly more general result (expressing the upper bound for the induced Tur\'{a}n number in terms of a Ramsey number involving $H$ -- see Corollary~\ref{ex} and Theorem~\ref{triangle}) followed by an improvement to the general upper bound.

\begin{restatable}{theorem}{exgeneral}\label{exgeneral}
	Fix a graph $H$ with $v_{H}$ vertices. For any integer $t \geqslant 1$,
	\begin{align*}
		\operatorname{ex}(n, \{H, K_{2, t + 1}\textnormal{-ind}\}) & < (t + 1)^{\frac{v_{H} - 1}{2}} n^{\frac{3}{2}}, \\
		\operatorname{ex}(n, \{H, K_{2, t + 1}\textnormal{-ind}\}) & < e^{\frac{v_{H}}{2} - 1} 2^{t - 1} n^{\frac{3}{2}}.
	\end{align*}
\end{restatable}

\noindent The first bound shows that, for non-bipartite $H$, the correct growth rate in $t$ is a polynomial in $t$ times $n^{3/2}$. The second bound is better when $t$ and $v_{H}$ are of comparable size.

\section{Notation, main result and organisation}

If $v$ is a vertex of a graph $G = (V, E)$ then $\Gamma(v) = \{u \in V : uv \in E\}$ is the neighbourhood of $v$. We set $G_{v} = G[\Gamma(v)]$. For a fixed graph $H$, let $\{H - x\}$ be the set of graphs obtained by removing a single vertex from $H$ and let $\{H - \bar{e}\}$ be the set of graphs obtained from $H$ by either removing a single vertex or two non-adjacent vertices. In particular the Ramsey number, $R(K_{t}, \{H - x\})$, is the least $n$ such that any red-blue colouring of the edges of $K_{n}$ contains either a red $K_{t}$ or a blue graph which can be obtained from $H$ by removing a single vertex.

Our main result is the following which applies for all values of $\alpha$.

\begin{restatable}{theorem}{main}\label{main}
	Fix a graph $H$. Let $G$ be a graph on $n$ vertices with $\alpha \binom{n}{2}$ edges containing no induced $K_{2, t}$ $(t \geqslant 2)$ and let $\beta = \beta_{t}(\alpha)$. 
	
	If $R(K_{t}, \{H - x\}) \leqslant \beta^{2} n$, then $H$ is a subgraph of $G$. In particular, if $R(K_{t}, \{H - x\}) \leqslant \frac{t - 1}{t^{2}} \cdot \alpha^{2} n$, then $H$ is a subgraph of $G$.
\end{restatable}

\noindent The sufficiency of $R(K_{t}, \{H - x\}) \leqslant \frac{t - 1}{t^{2}} \cdot \alpha^{2} n$ follows from the following lemma which relates $\beta$ to $\alpha$ in a manageable way.

\begin{lemma}\label{calc}
	For all $\alpha \in [0, 1]$ and integers $t \geqslant 2$, $\beta = \beta_{t}(\alpha)$ satisfies
	\begin{align*}
		& (t - 1)\bigl(\alpha - \beta^{2}\bigr)^{2} = t^{2} (1 - \alpha) \beta^{2}, \\
		& \tfrac{\sqrt{t - 1}}{t} \alpha \leqslant \beta \leqslant \alpha, \\
		& \beta \rightarrow 1, \text{ as } \alpha \rightarrow 1.		
	\end{align*}
\end{lemma}

\begin{proof}
	The equation $(t - 1)(\alpha - \beta^{2})^{2} = t^{2} (1 - \alpha) \beta^{2}$ is a quadratic in $\beta^{2}$. One can check that $\beta_{t}(\alpha)$ does indeed square to a solution of this quadratic.
	
	Fix $t$ and define the function $f(x) = \sqrt{1 - (1 - 2/t)^{2} x} - \sqrt{1 - x}$ for $x \in [0, 1]$. Then $f$ is convex increasing with $f(0) = 0$ and $f(1) = \frac{2 \sqrt{t - 1}}{t}$. Thus $f(x) \leqslant \frac{2 \sqrt{t - 1}}{t} x$. Also the derivative of $f$ at zero is $\frac{2}{t} - \frac{2}{t^{2}} = \frac{2(t - 1)}{t^{2}}$ so $f(x) \geqslant \frac{2(t - 1)}{t^{2}} x$. In particular $\beta = \frac{t}{2 \sqrt{t - 1}} f(\alpha)$ satisfies $\frac{\sqrt{t - 1}}{t} \alpha \leqslant \beta \leqslant \alpha$.
	
	Finally, $f$ is continuous so, as $\alpha$ tends to 1, $\beta$ tends to $\frac{t}{2 \sqrt{t - 1}} f(1) = 1$.
\end{proof}

\noindent We prove Theorem~\ref{main} in Section~\ref{proof}. Before that we use Ramsey estimates to obtain various corollaries. We normally give two versions of the results: one which holds for all values of $n$ and a stronger bound which holds for large enough $n$ (in terms of $\alpha$). The latter is only really applicable in the regime where $\alpha$ is bounded away from zero.

In Section~\ref{cliques} we look at the special case where $H$ is a complete graph, proving Theorems~\ref{cliquethm},~\ref{K23thm} and~\ref{K2tthm}. In Section~\ref{IndTur} we consider general $H$ for the Induced Tur\'{a}n problem (so $\alpha$ going to zero) and prove Theorem~\ref{exgeneral}. Finally in Section~\ref{Triangles} we exhibit a variation on our methods which gives a slight asymptotic improvement for the induced Tur\'{a}n number of $H$-free graphs with no induced $K_{2, t}$. This includes the observation that such graphs contain $\mathcal{O}(n^{27/14}) = o(n^{2})$ triangles.

\section{Clique numbers of graphs with no induced $K_{2, t}$}\label{cliques}

If we take $H = K_{r + 1}$ in Theorem~\ref{main} then $\{H - x\} = \{K_{r}\}$ so Theorem~\ref{cliquethm} is immediate.

\cliquethm*

\begin{remark}\label{remark}
	The following example illustrates why we believe this result is in a sense tight as $\alpha \rightarrow 1$. Consider a graph $G$ on $n$ vertices which has no independent $t$-set and smallest possible clique number (a Ramsey-like graph): that is, $R(t, \omega(G) + 1) > n \geqslant R(t, \omega(G))$. Now $G$ has no independent $t$-set so does not contain an induced $K_{2, t}$. If there are such graphs with $(1 - o(1)) \binom{n}{2}$ edges then these form a sequence of graphs for which $\alpha \rightarrow 1$ (and so $\beta \rightarrow 1$), but for which the statement becomes false if $\beta$ is actually replaced by 1.
	
	We do believe that such graphs have $(1 - o(1)) \binom{n}{2}$ edges. This would follow, for example, from $\frac{R(t - 1, m)}{R(t, m)} \rightarrow 0$ as $m \rightarrow \infty$ (true for $t = 3$ and 4 by standard Ramsey bounds but not known in general): the non-neighbours of a vertex in such a graph, $G$, cannot contain an independent $(t - 1)$-set, so there are at most $R(t - 1, \omega(G) + 1)$ non-neighbours, and so $\delta(G)$ would be $(1 - o(1))n$.
\end{remark}

\noindent The following corollary for $t = 3$ contains Theorem~\ref{K23thm}.

\begin{corollary}
	Let $G$ be a graph on $n$ vertices with $\alpha \binom{n}{2}$ edges which contains no induced $K_{2, 3}$. Let $\beta = \beta_{3}(\alpha) = \frac{3}{2 \sqrt{2}} \bigl[\sqrt{1 - \frac{\alpha}{9}} - \sqrt{1 - \alpha}\bigr]$. Then
	\begin{align*}
		\omega(G) & \geqslant \lfloor \beta \sqrt{2n} \rfloor \geqslant \bigl\lfloor \tfrac{2}{3} \alpha \sqrt{n} \bigr\rfloor \text{ for all } n, \text{ and} 
		\\
		\omega(G) & \geqslant \beta \sqrt{\tfrac{1}{2} n \log n} + 2 \geqslant \tfrac{1}{3} \alpha \sqrt{n \log n} + 2 \text{ for large enough }n, \text{ say } n \geqslant \exp(2e^{2} \beta^{-2}).
	\end{align*}
\end{corollary}

\begin{proof}
	Firstly, the theorem of Erd\H{o}s and Szekeres~\cite{ErdosSzekeres1935} gives that $R(3, r) \leqslant \binom{r + 1}{2}$ for all positive $r$. Thus $r = \lfloor \beta \sqrt{2n} \rfloor - 1$ satisfies $R(3, r) \leqslant \frac{1}{2} \lfloor \beta \sqrt{2n} \rfloor^{2} \leqslant \beta^{2} n$ and so Theorem~\ref{cliquethm} gives the first result.
	
	Secondly, $R(3, r) \leqslant \frac{(r - 2)^{2}}{\log(r - 1) - 1}$ for all $r \geqslant 4$ (a corollary of Shearer's result on independent sets in triangle-free graphs,~\cite{Shearer1983}). Thus $r = \Bigl\lfloor \beta \sqrt{\frac{1}{2} n \log n} \Bigr\rfloor + 2$ satisfies $R(3, r) \leqslant \beta^{2} n$ provided $n \geqslant \exp(2e^{2} \beta^{-2})$.
\end{proof}

\noindent The following corollary (which contains Theorem~\ref{K2tthm}) for $t$ larger than three is obtained in exactly the same way, using known bounds for $R(t, r)$. Improvements in the upper bounds on Ramsey numbers would improve the results.

\begin{corollary}\label{clique}
	Let $G$ be a graph on $n$ vertices with $\alpha \binom{n}{2}$ edges containing no induced $K_{2, t}$ and let $\beta = \beta_{t}(\alpha)$. Then
	\begin{align*}
		\omega(G) & \geqslant \bigl\lfloor \tfrac{t - 1}{e} (\beta^{2} n)^{\frac{1}{t - 1}}\bigr\rfloor - t + 3 \text{ and } \omega(G) \geqslant \bigl\lfloor \tfrac{t - 1}{4} (\alpha^{2} n)^{\frac{1}{t - 1}}\bigr\rfloor - t + 3 \text{ for all } n, \text{ and} \\
		\omega(G) & \geqslant \tfrac{1}{20} \bigl(\beta^{2} n\bigr)^{\frac{1}{t - 1}} \bigl(\tfrac{\log n}{t - 1}\bigr)^{1 - \frac{1}{t - 1}} \geqslant \tfrac{1}{20 t} \bigl(\alpha^{2} n (\log n)^{t - 2}\bigr)^{\frac{1}{t - 1}} \text{ for large enough } n \text{ in terms of } \beta.
	\end{align*}
\end{corollary}

\begin{proof}
	The theorem of Erd\H{o}s and Szekeres~\cite{ErdosSzekeres1935} gives that $R(t, r) \leqslant \binom{r + t - 2}{t - 1} \leqslant \frac{(r + t - 2)^{t - 1}}{(t - 1)!}$ for all positive $r$. Thus $r = \bigl\lfloor \bigl(\beta^{2} n (t - 1)!\bigr)^{\frac{1}{t - 1}}\bigr\rfloor - t + 2$ has $R(t, r) \leqslant \beta^{2} n$ so, by Theorem~\ref{cliquethm},
	\begin{equation*}
		\omega(G) \geqslant \bigl\lfloor \bigl(\beta^{2} n (t - 1)!\bigr)^{\frac{1}{t - 1}}\bigr\rfloor - t + 3 \geqslant \bigl\lfloor \bigl(\tfrac{t - 1}{t^{2}} \alpha^{2} n (t - 1)!\bigr)^{\frac{1}{t - 1}}\bigr\rfloor - t + 3.
	\end{equation*}
	Furthermore $(t - 1)! \geqslant \bigl(\frac{t - 1}{e}\bigr)^{t - 1}$ so $\bigl((t - 1)!\bigr)^{\frac{1}{t - 1}} \geqslant \frac{t - 1}{e}$. That $\bigl(\frac{t - 1}{t^{2}} (t - 1)!\bigr)^{\frac{1}{t - 1}} \geqslant \frac{t - 1}{4}$ follows from $(t - 1)! \geqslant \frac{(t - 1)^{t - 1/2}}{e^{t - 1}}$ for $t \geqslant 4$ and can be checked directly for $t = 2, 3$.
	
	Finally $R(t, r) \leqslant 2(20)^{t - 3} \frac{r^{t - 1}}{(\log r)^{t - 2}}$ for $r$ sufficiently large (see Bollob\'{a}s~\cite[Thm 12.17]{Bollobas2001RandomGraphs}) so we obtain, for all large $n$, that
	\begin{equation*}
		\omega(G) \geqslant \frac{1}{20} \bigg(\frac{\beta^{2} n (\log n)^{t - 2}}{(t - 1)^{t - 2}}\bigg)^{\frac{1}{t - 1}} \geqslant \frac{1}{20} \bigg(\frac{\alpha^{2} n (\log n)^{t - 2}}{t^{2} (t - 1)^{t - 3}}\bigg)^{\frac{1}{t - 1}}. \qedhere
	\end{equation*}
\end{proof}

\section{Tur\'{a}n number for no $H$ and no induced $K_{2, t}$}\label{IndTur}

We now focus on the regime where $\alpha$ goes to zero and consider the induced Tur\'{a}n numbers introduced by Loh, Tait, Timmons and Zhou.

\begin{corollary}\label{ex}
	Fix a graph $H$. For any integer $t \geqslant 2$,
	\begin{equation*}
		\operatorname{ex}(n, \{H, K_{2, t}\textnormal{-ind}\}) < \tfrac{t}{2 \sqrt{t - 1}} R(K_{t}, \{H - x\})^{\frac{1}{2}} n^{\frac{3}{2}}.
	\end{equation*}
\end{corollary}

\begin{proof}
	Let $G$ be a graph on $n$ vertices containing no induced $K_{2, t}$ and no copy of $H$. By Theorem~\ref{main}, $R(K_{t}, \{H - x\}) > \frac{t - 1}{t^{2}} \cdot \alpha^{2} n$ so $\alpha < \frac{t}{\sqrt{t - 1}} n^{-\frac{1}{2}} R(K_{t}, \{H - x\})^{\frac{1}{2}}$. Therefore
	\begin{equation*}
		e(G) = \alpha \tbinom{n}{2} < \tfrac{t}{2 \sqrt{t - 1}} R(K_{t}, \{H - x\})^{\frac{1}{2}} n^{\frac{1}{2}} (n - 1). \qedhere
	\end{equation*}
\end{proof}

\noindent We now use Theorem~\ref{K2tthm} and Corollary~\ref{ex} to prove Theorem~\ref{exgeneral}, restated here for convenience.

\exgeneral*

\begin{proof}
	Note that $R(K_{t}, \{H - x\}) \leqslant R(t + 1, v_{H} - 1)$. For all positive integers $a$ and $b$
	\begin{equation*}
		\tbinom{a + b - 2}{a - 1} = \tfrac{a + b - 2}{a - 1} \cdot \tfrac{a + b - 3}{a - 2} \ \dotsm \ \tfrac{b}{1} \leqslant b^{a - 1},
	\end{equation*}
	and so Erd\H{o}s and Szekeres's bound~\cite{ErdosSzekeres1935} gives $R(K_{t + 1}, \{H - x\}) \leqslant (t + 1)^{v_{H} - 2}$. By Corollary~\ref{ex},
	\begin{equation*}
		\operatorname{ex}(n, \{H, K_{2, t + 1}\textnormal{-ind}\}) < \tfrac{t + 1}{2 \sqrt{t}} (t + 1)^{\frac{v_{H}}{2} - 1} n^{\frac{3}{2}} < (t + 1)^{\frac{v_{H} - 1}{2}} n^{\frac{3}{2}}.
	\end{equation*}
	
	Let $G$ be a graph on $n$ vertices with $\alpha \binom{n}{2}$ edges and no induced $K_{2, t + 1}$. If $G$ does not contain $H$ then $\omega(G) < v_{H}$ so, by Theorem~\ref{K2tthm}, $v_{H} > \bigl\lfloor \tfrac{t}{4} (\alpha^{2} n)^{\frac{1}{t}}\bigr\rfloor - t + 2$. $v_{H} + t - 2$ is an integer so
	\begin{equation*}
		v_{H} + t - 2  > \tfrac{t}{4} (\alpha^{2} n)^{\frac{1}{t}}.
	\end{equation*}
	Now rearranging and using $e(G) = \alpha \binom{n}{2} < \frac{\alpha}{2} n^{2}$ we get
	\begin{equation*}
		e(G) < n^{\frac{3}{2}} 2^{t - 1} \bigl(1 + \tfrac{v_{H} - 2}{t}\bigr)^{\frac{t}{2}} < e^{\frac{v_{H}}{2} - 1} 2^{t - 1} n^{\frac{3}{2}}. \qedhere 
	\end{equation*}
\end{proof}

\section{Proof of main result}\label{proof}

For convenience we restate the main result here. As mentioned earlier, the proof is motivated by that of Holmsen~\cite{Holmsen2019}.

\main*

\begin{proof}
	By Lemma~\ref{calc}, for $\alpha \in [0, 1]$ we have $0 \leqslant \beta \leqslant \alpha \leqslant 1$ and also $\frac{t - 1}{t^{2}}\bigl(\alpha - \beta^{2}\bigr)^{2} = (1 - \alpha) \beta^{2}$.
	
	Suppose that $G$ does not contain $H$. Let the set of missing edges in $G$ be $M = \binom{V(G)}{2} - E(G)$, which has size $(1 - \alpha)\binom{n}{2}$. For each $v \in V(G)$, let
	\begin{align*}
		m_{v} & \text{ be the total number of missing edges in } G_{v}, \\
		\bar{\Delta}_{1}&, \dotsc, \bar{\Delta}_{\gamma_{v}} \text{ be a maximal collection of pairwise vertex-disjoint}\\
		& \qquad \qquad \qquad \text{independent } t\text{-sets in } G_{v}.
	\end{align*}
	By the maximality of $\gamma_{v}$, $G[\Gamma(v)\backslash \cup _{j} \bar{\Delta}_{j}]$ does not contain an independent $t$-set. Furthermore it does not contain any $H - x$ (else together with $v$ we have a copy of $H$ in $G$). Thus
	\begin{align}
		& R(K_{t}, \{H - x\}) - 1 \geqslant \lvert \Gamma(v) \rvert - t \gamma_{v} = \deg(v) - t \gamma_{v}, \text{ and so}\nonumber \\
		& \gamma_{v} \geqslant \tfrac{1}{t}[\deg(v) - R(K_{t}, \{H - x\}) + 1] \geqslant \tfrac{1}{t} [\deg(v) - \beta^{2} (n - 1)]. \label{gammav}
	\end{align}
	$G$ contains no induced $K_{2, t}$ so at most one vertex in $\bar{\Delta}_{i}$ is adjacent to all of $\bar{\Delta}_{j}$ (for any $i \neq j$). In particular, between $\bar{\Delta}_{i}$ and $\bar{\Delta}_{j}$ there must be at least $t - 1$ missing edges. These missing edges are in no $\bar{\Delta}_{k}$ (by vertex-disjointness) and each such edge corresponds to only one pair $(\bar{\Delta}_{i}, \bar{\Delta}_{j})$. Considering these missing edges as well as the ones contained entirely in each $\bar{\Delta}_{k}$ gives
	\begin{equation*}
		m_{v} \geqslant \tbinom{t}{2} \gamma_{v} + (t - 1) \tbinom{\gamma_{v}}{2} = q(\gamma_{v}),
	\end{equation*}
	where
	\begin{equation*}
		q(x) = \tfrac{t - 1}{2} \cdot x (x + t - 1)
	\end{equation*}
	is convex and increasing for non-negative $x$. Averaging \eqref{gammav} over $v \in G$ we have
	\begin{equation*}
		\frac{1}{n} \sum_{v \in G} \gamma_{v} \geqslant \tfrac{1}{tn} [2e(G) - \beta^{2} n (n - 1)] = \tfrac{1}{t} \bigl(\alpha - \beta^{2}\bigr) (n - 1).
	\end{equation*}
	Using Jensen, the monotonicity of $q$, and the fact that $\alpha \geqslant \beta \geqslant \beta^{2}$ gives
	\begin{align*}
		\frac{1}{n} \sum_{v \in G} m_{v} & \geqslant \frac{1}{n} \sum_{v \in G} q(\gamma_{v}) \geqslant q \biggl(\frac{1}{n} \sum_{v \in G} \gamma_{v}\biggr) \geqslant q \bigl(\tfrac{1}{t} \bigl(\alpha - \beta^{2}\bigr) (n - 1)\bigr) \\
		& = \tfrac{t - 1}{2} \cdot \tfrac{1}{t} \bigl(\alpha - \beta^{2}\bigr)(n - 1) \cdot \bigl(\tfrac{1}{t}\bigl(\alpha - \beta^{2}\bigr) (n - 1) + t - 1\bigr) \\
		& \geqslant \tfrac{t - 1}{2} \cdot \tfrac{1}{t} \bigl(\alpha - \beta^{2}\bigr) (n - 1) \cdot \tfrac{1}{t} \bigl(\alpha - \beta^{2}\bigr) n \\
		& = \tfrac{t - 1}{t^{2}} \bigl(\alpha - \beta^{2}\bigr)^{2} \cdot \tbinom{n}{2} \\
		& = \beta^{2} (1 - \alpha) \tbinom{n}{2}.
	\end{align*}
	Now $\sum_{v \in G} m_{v} = \sum_{\bar{e} \in M} \# \{v \text{ with } \bar{e} \subset \Gamma(v)\}$ and $\lvert M \rvert = (1 - \alpha) \binom{n}{2}$ so there is $\bar{e} \in M$ and $S \subset V(G)$ of size at least $\beta^{2}n$ such that $\bar{e} \subset \Gamma(v)$ for each $v \in S$: that is, all vertices of $S$ are in the common neighbourhood of the two end-vertices of the missing edge $\bar{e}$.
	
	Now $G[S]$ contains no independent $t$-set (else together with $\bar{e}$ we have an induced $K_{2, t}$) and $\lvert S \rvert \geqslant \beta^{2}n \geqslant R(K_{t}, \{H - x\})$ so $G[S]$ contains a copy of some $H - x$. Together with one end-vertex of $\bar{e}$ we have a copy of $H$ in $G$.
\end{proof}

\begin{remark}
	It is natural to ask whether the ideas of this argument could be extended to graphs which contain no induced $K_{s, t}$. The argument above is so clean partly because the number of independent 2-sets in $G$ is determined by $\alpha$ (it is $\lvert M \rvert = (1 - \alpha)\binom{n}{2}$). Extending to no induced $K_{s, t}$ would require some knowledge of the number of independent $s$-sets in $G$.
\end{remark}

\section{Improvement when there are few triangles}\label{Triangles}

Corollary~\ref{ex} says $\operatorname{ex}(n, \{H, K_{2, t}\textnormal{-ind}\}) < \frac{t}{2 \sqrt{t - 1}} R(K_{t}, \{H - x\})^{\frac{1}{2}} n^{\frac{3}{2}}$. In this section we show that $n$-vertex $H$-free graphs with no induced $K_{2, t}$ contain $o(n^{2})$ triangles. This asymptotically improves our lower bound on the number of missing edges in each neighbourhood and so improves Corollary~\ref{ex} by a factor of $\sqrt{t}$ as well as reducing the Ramsey number used -- see Theorem~\ref{triangle}.

\begin{theorem}\label{fewtriangles}
	Fix a graph $H$ and an integer $t \geqslant 2$. Every $n$-vertex graph which contains no copy of $H$ and no induced $K_{2, t}$ has at most $\mathcal{O}(n^{27/14})$ triangles.
\end{theorem}

\begin{proof}
	By Corollary~\ref{ex}, there is a constant $C = C_{H, t}$ such that every $m$-vertex graph which contains no copy of $H$ and no induced $K_{2, t}$ has at most $Cm^{3/2}$ edges. 
	
	Let $G$ be a graph on $n$ vertices containing no induced $K_{2, t}$ and no copy of $H$. For each vertex $v$ of $G$, note that exactly $e(G_{v})$ triangles in $G$ contain $v$. As $G$ has no copy of $H$ and no induced $K_{2, t}$,
	\begin{align*}
		e(G) & \leqslant Cn^{3/2}, \\
		e(G_{v}) & \leqslant C \deg(v)^{3/2}.
	\end{align*}
	Let $X$ be the set of vertices in $G$ whose degree is at least $f(n)$ (a function of $n$ whose value we give later). Firstly,
	\begin{equation*}
		\lvert X \rvert f(n) \leqslant \sum_{v \in X} \deg(v) \leqslant 2e(G) \leqslant 2C n^{3/2},
	\end{equation*}
	and so the number of triangles in $G$ whose vertices are all in $X$ is at most
	\begin{equation*}
		\tbinom{\lvert X \rvert}{3} \leqslant \tfrac{1}{6} \lvert X \rvert^{3} \leqslant \tfrac{4}{3} C^{3} n^{9/2} f(n)^{-3}.
	\end{equation*}
	The number of triangles of $G$ containing at least one vertex in $V(G) \setminus X$ is at most
	\begin{equation*}
		\sum_{v \not \in X} e(G_{v}) \leqslant C \sum_{v \not \in X} \deg(v)^{3/2}.
	\end{equation*}
	The function $x \mapsto x^{3/2}$ is convex and all $v \not \in X$ satisfy $\deg(v) \leqslant f(n)$, so
	\begin{align*}
		\sum_{v \not \in X} \deg(v)^{3/2} & \leqslant \biggl(f(n)^{-1} \sum_{v \not \in X} \deg(v)\biggr) f(n)^{3/2} = f(n)^{1/2} \sum_{v \not \in X} \deg(v) \\
		& \leqslant  2f(n)^{1/2} e(G) \leqslant 2C n^{3/2} f(n)^{1/2}.
	\end{align*}
	Thus, the number of triangles in $G$ is at most
	\begin{equation*}
		\tfrac{4}{3} C^{3} n^{9/2} f(n)^{-3} + 2C^{2} n^{3/2} f(n)^{1/2}.
	\end{equation*}
	We minimise this last expression by taking $f(n) = 2^{4/7} C^{2/7} n^{6/7}$ which gives a value less than $3 C^{15/7} n^{27/14}$.
\end{proof}

\begin{theorem}\label{triangle}
	Fix a graph $H$ and an integer $t \geqslant 2$. Let $\Delta(n, H, t)$ denote the greatest number of triangles in a graph on $n$ vertices containing no copy of $H$ and no induced $K_{2, t}$. Let $G$ be a graph on $n$ vertices with $\alpha \binom{n}{2}$ edges containing no induced $K_{2, t}$. If
	\begin{equation*}
		\alpha^{2}(n - 1) > R(K_{t}, \{H - \bar{e}\}) - 1 + 3 \Delta(n, H, t) \tbinom{n}{2}^{-1},
	\end{equation*}
	then $H$ is a subgraph of $G$. In particular,
	\begin{equation*}
		\operatorname{ex}(n, \{H, K_{2, t}\textnormal{-ind}\}) \leqslant \tfrac{1}{2} \bigl(R(K_{t}, \{H - \bar{e}\}) - 1 + o(1)\bigr)^{\frac{1}{2}} n^{\frac{3}{2}}.
	\end{equation*}
\end{theorem}

\begin{proof}
	$R(K_{t}, \{H - \bar{e}\}) \geqslant 2$ so we in fact have
	\begin{equation*}
		\alpha[\alpha(n - 1) - 1] > (1 - \alpha)(R(K_{t}, \{H - \bar{e}\}) - 1) + 3 \Delta(n, H, t) \tbinom{n}{2}^{-1}.
	\end{equation*}
	We will use this to show $H$ is a subgraph of $G$. Suppose for contradiction it is not. Let the set of missing edges in $G$ be $M = \binom{V(G)}{2} - E(G)$ which has size $(1 - \alpha) \binom{n}{2}$. For each $v \in V(G)$ let
	\begin{align*}
		e_{v} & = e(G_{v}), \\
		m_{v} & = \text{ total number of missing edges in } G_{v}.
	\end{align*}
	First note that $e_{v} + m_{v} = \binom{\lvert\Gamma(v)\rvert}{2} = \binom{\deg(v)}{2}$, so, by Jensen's inequality,
	\begin{equation*}
		\sum_{v \in G} (m_{v} + e_{v}) \geqslant n \tbinom{2e(G)/n}{2} = n \tbinom{\alpha(n - 1)}{2} = \alpha [\alpha(n - 1) - 1] \tbinom{n}{2}.
	\end{equation*}
	
	Now $e_{v}$ is also the number of triangles in $G$ containing $v$ so $\sum_{v \in G} e_{v}$ is three times the number of triangles in $G$ which is at most $3 \Delta(n, H, t)$. Thus
	\begin{equation*}
		\sum_{v \in G} m_{v} \geqslant \alpha [\alpha(n - 1) - 1] \tbinom{n}{2} - 3 \Delta(n, H, t) > (1 - \alpha)\tbinom{n}{2} (R(K_{t}, \{H - \bar{e}\}) - 1).
	\end{equation*}
	Now $\sum_{v \in G} m_{v} = \sum_{\bar{e} \in M} \# \{v \text{ with } \bar{e} \subset \Gamma(v)\}$ and $\lvert M \rvert = (1 - \alpha) \binom{n}{2}$ so there is some missing edge $\bar{e}$ and some $S \subset V(G)$ of size $R(K_{t}, \{H - \bar{e}\})$ with $\bar{e} \subset \Gamma(v)$ for each $v \in S$. $G[S]$ does not contain an independent $t$-set (else together with $\bar{e}$ we have an induced $K_{2, t}$ in $G$) so $G[S]$ contains a copy of some $H - x$ or some $H - \bar{e}$. Together with $\bar{e}$ we have that $G$ contains a copy of $H$ proving the first result.
	
	By Theorem~\ref{fewtriangles}, $\Delta(n, H, t) = o(n^{2})$. Suppose that $G$ is a graph on $n$ vertices with no $H$ and no induced $K_{2, t}$. We must have
	\begin{equation*}
		\alpha^{2}(n - 1) \leqslant R(K_{t}, \{H - \bar{e}\}) - 1 + o(1).
	\end{equation*}
	Using $e(G) = \alpha \binom{n}{2}$ we get the required result.
\end{proof}

\noindent\textbf{Acknowledgement.} The author would like to thank Andrew Thomason for many helpful discussions as well as the anonymous referees for their careful reading of the manuscript and bringing~\cite{EGM2019} to the author's attention.



\begin{thebibliography}{10}
	\bibitem{Bollobas2001RandomGraphs}
	B.~Bollob\'{a}s.
	\newblock {\em Random Graphs}.
	\newblock Cambridge Studies in Advanced Mathematics. Cambridge University
	Press, 2nd edition, 2001.
	
	\bibitem{ErdosSimonovtis1966}
	P.~Erd\H{o}s and M.~Simonovits.
	\newblock A limit theorem in graph theory.
	\newblock {\em Studia Scientiarum Mathematicarum Hungarica}, 1:51--57, 1966.
	
	\bibitem{ErdosStone1946}
	P.~Erd\H{o}s and A.~H. Stone.
	\newblock On the structure of linear graphs.
	\newblock {\em Bulletin of American Mathematical Society}, 52:1087--1091, 1946.
	
	\bibitem{ErdosSzekeres1935}
	P.~Erd\H{o}s and G.~Szekeres.
	\newblock A combinatorial problem in geometry.
	\newblock {\em Compositio Mathematica}, 2:463--470, 1935.
	
	\bibitem{EGM2019}
	B.~Ergemlidze, E.~Gy\H{o}ri, and A.~Methuku.
	\newblock Tur{\'a}n number of an induced complete bipartite graph plus an odd
	cycle.
	\newblock {\em Combinatorics, Probability and Computing}, 28(2):241--252,
	2019.
	
	\bibitem{Furedi1996}
	Z.~F{\"u}redi.
	\newblock New asymptotics for bipartite tur\'{a}n numbers.
	\newblock {\em Journal of Combinatorial Theory, Series A}, 75(1):141--144,
	1996.
	
	\bibitem{GHS2002}
	A.~Gy{\'a}rf{\'a}s, A.~Hubenko, and J.~Solymosi.
	\newblock Large cliques in ${C}_{4}$-free graphs.
	\newblock {\em Combinatorica}, 22(2):269--274, Apr 2002.
	
	\bibitem{Holmsen2019}
	A.~Holmsen.
	\newblock Large cliques in hypergraphs with forbidden substructures.
	\newblock {\em Combinatorica}, 40(4):527--537, Mar 2020.
	
	\bibitem{LTTZ2017}
	P.-S. Loh, M.~Tait, C.~Timmons, and R.~M. Zhou.
	\newblock Induced Tur\'{a}n numbers.
	\newblock {\em Combinatorics Probability and Computing}, 27:1--15, Nov 2017.
	
	\bibitem{Shearer1983}
	J.~B. Shearer.
	\newblock A note on the independence number of triangle-free graphs.
	\newblock {\em Discrete Mathematics}, 46(1):83--87, 1983.
\end{thebibliography}
\end{document}